\documentclass[a4paper,11pt]{amsart}
\usepackage{color}
\usepackage{amsmath,amscd,amsthm,upref}
\usepackage[psamsfonts]{amssymb}
\usepackage{bm}
\usepackage[all]{xy}
\usepackage{enumerate}
\usepackage{url}

\renewcommand{\labelenumi}{\rm (\arabic{enumi})}
\numberwithin{equation}{section}

\newtheorem{thm}{Theorem}[]

\newtheorem{lmm}[thm]{Lemma}

\theoremstyle{definition}

\theoremstyle{remark}

\DeclareMathOperator{\Diff}{\bf Diff}
\DeclareMathOperator{\Top}{\bf Top}
\DeclareMathOperator{\NumG}{\bf NG}
\DeclareMathOperator{\STop}{\bf STop}
\DeclareMathOperator{\Ho}{\bf Ho}
\DeclareMathOperator{\smap}{\bf smap}
\DeclareMathOperator{\Cinfty}{{\mathit C}^{\infty}}
\DeclareMathOperator{\id}{id}
\DeclareMathOperator{\Sd}{Sd}
\DeclareMathOperator{\Int}{Int}
\DeclareMathOperator{\Loops}{Loops}
\DeclareMathOperator{\Fr}{Fr}
\DeclareMathOperator*{\colim}{colim}

\newcommand{\norm}[1]{\|{#1}\|}
\newcommand{\blck}[1]{\langle{#1}\rangle}
\newcommand{\R}{\mathbf{R}}
\newcommand{\I}{\mathcal{I}}
\newcommand{\J}{\mathcal{J}}
\newcommand{\K}{\mathcal{K}}
\newcommand{\semicolon}{\,;}
\newcommand{\bI}{\partial I}
\newcommand{\tI}{\tilde{I}}
\newcommand{\abs}[1]{\lvert{#1}\rvert}
\newcommand{\rel}[1]{\ \mbox{\rm rel}\ #1}

\title{The Baumkuchen Theorem}
\author[M. Ando]{Masanori Ando} %
\address{Faculty of Education for Human Growth \\ Nara Gakuen
  University \\ Nara 636-8503 \\ Japan} %
\email{m-ando@naragakuen-u.jp} 
\author[T. Haraguchi]{Tadayuki Haraguchi} %
\address{Faculty of Education for Human Growth \\ Nara Gakuen
  University \\ Nara 636-8503 \\ Japan} %
\email{t-haraguchi@naragakuen-u.jp} %
\date{\today} %
\subjclass[$2020$]{Primary 51M04 ; Secondary 51M05} %
\begin{document}
\maketitle
%%%%% Abstract
\begin{abstract}
%In this paper we will give the Baumkuchen Lemma (see Lemma \ref{Baumkuchen Lemma}).
In this paper,
we present \emph{the Baumkuchen Theorem} (cf.~Theorem \ref{Baumkuchen Theorem}) 
related to the combination of divided Baumkuchen pieces.
It can be proved using the basic properties of elementary geometry.
We also apply some lemmas to prove \emph{the Pizza Theorem} (cf.~Theorem \ref{thm:Pizza Theorem}).
%which is expressed in a statement similar to the claim of the Pizza Theorem \cite{Upton}.
%In the process of proving this theorem,
%we introduce one of the proof of \emph{the Pizza Theorem} (cf.~Theorem \ref{thm:Pizza Theorem}).
%By applying this theorem,
%we introduce one proof of the Pizza Theorem (cf.~Theorem \ref{thm:Pizza Theorem}).
%In particular,
%we prove them by using only the basic properties of elementary geometry.
%We will introduce one proof of the Pizza Theorem \cite{Upton} by utilizing only the basic properties of elementary geometry.
\end{abstract}
%%%%%
%
%使うかもしれない
%
%
%
%The Pizza Theorem was posed by L.~J.~Upton in Mathematics Magazine almost fifty years ago (cf.~\cite{Upton}).
%In \cite{Goldberg},
%Goldberg first proved this problem and the results are sprinkled throughout the literature (cf.~\cite{Berz1}, \cite{Berz2}, \cite{Hart}, \cite{Hess}, \cite{Hirs}, \cite{Mabry}, \cite{Nelsen} and \cite{Pearce}).
%Introduce the statement of this theorem (see {\bf Figure 1}).
%Throughout this paper, we will discuss under the following conditions.
%Let $P$ be a point of a disk $D$, and  let $n$ be a multiple of $4$. 
%Then divide a disk $D$ into $2n$ slices by making $n$ straight lines through $P$, where the lines meeting to form $2n$ equal angles.
%Number the slices consecutively in a clockwise fashion. 
%Then we have the following.
%\begin{thm}[Pizza Theorem \cite{Upton}]\label{thm:Pizza Theorem}
%The sum of the areas of the odd-numbered slices equals the sum of the areas of the even-numbered slices.
%\end{thm}
%
%
%
%study
First,
we present the statement of \emph{the Baumkuchen Theorem}.
%$\\[0.1cm]$
%
%
%
%
%
%
%
\[
%WinTpicVersion4.32a
{\unitlength 0.1in%
\begin{picture}(45.9000,26.5000)(5.9000,-36.7000)%
% CIRCLE 2 0 3 0 Black White  
% 4 1610 2220 1610 1200 1610 1200 1610 1200
% 
\special{pn 8}%
\special{ar 1610 2220 1020 1020 0.0000000 6.2831853}%
% CIRCLE 2 0 3 0 Black White  
% 4 4190 2200 4190 1210 4190 1210 4190 1210
% 
\special{pn 8}%
\special{ar 4190 2200 990 990 0.0000000 6.2831853}%
% CIRCLE 2 0 3 0 Black White  
% 4 1610 2200 1610 1600 1610 1600 1610 1600
% 
\special{pn 8}%
\special{ar 1610 2200 600 600 0.0000000 6.2831853}%
% CIRCLE 2 0 3 0 Black White  
% 4 4190 2210 4190 1610 4190 1610 4190 1610
% 
\special{pn 8}%
\special{ar 4190 2210 600 600 0.0000000 6.2831853}%
% STR 2 0 3 0 Black White  
% 4 4210 2010 4210 2110 2 0 0 0
% $\scriptstyle O$
\put(42.1000,-21.1000){\makebox(0,0)[lb]{$\scriptstyle O$}}%
% STR 2 0 3 0 Black White  
% 4 4200 2110 4200 2210 2 0 0 0
% $\bullet$
\put(42.0000,-22.1000){\makebox(0,0)[lb]{$\bullet$}}%
% STR 2 0 3 0 Black White  
% 4 4200 2110 4200 2210 2 0 0 0
% $\bullet$
\put(42.0000,-22.1000){\makebox(0,0)[lb]{$\bullet$}}%
% STR 2 0 3 0 Black White  
% 4 1610 2110 1610 2210 2 0 0 0
% $\bullet$
\put(16.1000,-22.1000){\makebox(0,0)[lb]{$\bullet$}}%
% STR 2 0 3 0 Black White  
% 4 1620 2010 1620 2110 2 0 0 0
% $\scriptstyle O$
\put(16.2000,-21.1000){\makebox(0,0)[lb]{$\scriptstyle O$}}%
% LINE 0 0 3 0 Black White  
% 100 1490 1210 600 2100 1560 1200 590 2170 1620 1200 590 2230 1680 1200 590 2290 1730 1210 600 2340 1080 1920 610 2390 1040 2020 620 2440 1020 2100 630 2490 1010 2170 650 2530 1010 2230 660 2580 1020 2280 680 2620 1020 2340 700 2660 1040 2380 720 2700 1050 2430 740 2740 1070 2470 760 2780 1090 2510 790 2810 1120 2540 810 2850 1140 2580 840 2880 1170 2610 870 2910 1200 2640 900 2940 1230 2670 930 2970 1270 2690 960 3000 1300 2720 1000 3020 1340 2740 1030 3050 1380 2760 1060 3080 1430 2770 1100 3100 1480 2780 1140 3120 1530 2790 1180 3140 1580 2800 1220 3160 1640 2800 1270 3170 1710 2790 1310 3190 1790 2770 1360 3200 1890 2730 1400 3220 2620 2060 1450 3230 2620 2120 1510 3230 2630 2170 1560 3240 2630 2230 1620 3240 2630 2290 1680 3240 2620 2360 1750 3230 2600 2440 1830 3210 2580 2520 1910 3190 2550 2610 2000 3160 2480 2740 2130 3090 2610 2010 2140 2480 2590 1970 2180 2380 2580 1920 2200 2300 2560 1880 2210 2230 2550 1830 2210 2170 2530 1790 2200 2120 2510 1750 2200 2060
% 
\special{pn 20}%
\special{pa 1490 1210}%
\special{pa 600 2100}%
\special{fp}%
\special{pa 1560 1200}%
\special{pa 590 2170}%
\special{fp}%
\special{pa 1620 1200}%
\special{pa 590 2230}%
\special{fp}%
\special{pa 1680 1200}%
\special{pa 590 2290}%
\special{fp}%
\special{pa 1730 1210}%
\special{pa 600 2340}%
\special{fp}%
\special{pa 1080 1920}%
\special{pa 610 2390}%
\special{fp}%
\special{pa 1040 2020}%
\special{pa 620 2440}%
\special{fp}%
\special{pa 1020 2100}%
\special{pa 630 2490}%
\special{fp}%
\special{pa 1010 2170}%
\special{pa 650 2530}%
\special{fp}%
\special{pa 1010 2230}%
\special{pa 660 2580}%
\special{fp}%
\special{pa 1020 2280}%
\special{pa 680 2620}%
\special{fp}%
\special{pa 1020 2340}%
\special{pa 700 2660}%
\special{fp}%
\special{pa 1040 2380}%
\special{pa 720 2700}%
\special{fp}%
\special{pa 1050 2430}%
\special{pa 740 2740}%
\special{fp}%
\special{pa 1070 2470}%
\special{pa 760 2780}%
\special{fp}%
\special{pa 1090 2510}%
\special{pa 790 2810}%
\special{fp}%
\special{pa 1120 2540}%
\special{pa 810 2850}%
\special{fp}%
\special{pa 1140 2580}%
\special{pa 840 2880}%
\special{fp}%
\special{pa 1170 2610}%
\special{pa 870 2910}%
\special{fp}%
\special{pa 1200 2640}%
\special{pa 900 2940}%
\special{fp}%
\special{pa 1230 2670}%
\special{pa 930 2970}%
\special{fp}%
\special{pa 1270 2690}%
\special{pa 960 3000}%
\special{fp}%
\special{pa 1300 2720}%
\special{pa 1000 3020}%
\special{fp}%
\special{pa 1340 2740}%
\special{pa 1030 3050}%
\special{fp}%
\special{pa 1380 2760}%
\special{pa 1060 3080}%
\special{fp}%
\special{pa 1430 2770}%
\special{pa 1100 3100}%
\special{fp}%
\special{pa 1480 2780}%
\special{pa 1140 3120}%
\special{fp}%
\special{pa 1530 2790}%
\special{pa 1180 3140}%
\special{fp}%
\special{pa 1580 2800}%
\special{pa 1220 3160}%
\special{fp}%
\special{pa 1640 2800}%
\special{pa 1270 3170}%
\special{fp}%
\special{pa 1710 2790}%
\special{pa 1310 3190}%
\special{fp}%
\special{pa 1790 2770}%
\special{pa 1360 3200}%
\special{fp}%
\special{pa 1890 2730}%
\special{pa 1400 3220}%
\special{fp}%
\special{pa 2620 2060}%
\special{pa 1450 3230}%
\special{fp}%
\special{pa 2620 2120}%
\special{pa 1510 3230}%
\special{fp}%
\special{pa 2630 2170}%
\special{pa 1560 3240}%
\special{fp}%
\special{pa 2630 2230}%
\special{pa 1620 3240}%
\special{fp}%
\special{pa 2630 2290}%
\special{pa 1680 3240}%
\special{fp}%
\special{pa 2620 2360}%
\special{pa 1750 3230}%
\special{fp}%
\special{pa 2600 2440}%
\special{pa 1830 3210}%
\special{fp}%
\special{pa 2580 2520}%
\special{pa 1910 3190}%
\special{fp}%
\special{pa 2550 2610}%
\special{pa 2000 3160}%
\special{fp}%
\special{pa 2480 2740}%
\special{pa 2130 3090}%
\special{fp}%
\special{pa 2610 2010}%
\special{pa 2140 2480}%
\special{fp}%
\special{pa 2590 1970}%
\special{pa 2180 2380}%
\special{fp}%
\special{pa 2580 1920}%
\special{pa 2200 2300}%
\special{fp}%
\special{pa 2560 1880}%
\special{pa 2210 2230}%
\special{fp}%
\special{pa 2550 1830}%
\special{pa 2210 2170}%
\special{fp}%
\special{pa 2530 1790}%
\special{pa 2200 2120}%
\special{fp}%
\special{pa 2510 1750}%
\special{pa 2200 2060}%
\special{fp}%
% LINE 0 0 3 1 Black White  
% 50 2490 1710 2180 2020 2470 1670 2170 1970 2440 1640 2150 1930 2410 1610 2130 1890 2390 1570 2100 1860 2360 1540 2080 1820 2330 1510 2050 1790 2300 1480 2020 1760 2270 1450 1990 1730 2240 1420 1950 1710 2200 1400 1920 1680 2170 1370 1880 1660 2130 1350 1840 1640 2090 1330 1790 1630 2050 1310 1740 1620 2010 1290 1690 1610 1970 1270 1640 1600 1920 1260 1580 1600 1880 1240 1510 1610 1830 1230 1430 1630 1780 1220 1330 1670 1420 1220 610 2030 1340 1240 630 1950 1250 1270 660 1860 1140 1320 720 1740
% 
\special{pn 20}%
\special{pa 2490 1710}%
\special{pa 2180 2020}%
\special{fp}%
\special{pa 2470 1670}%
\special{pa 2170 1970}%
\special{fp}%
\special{pa 2440 1640}%
\special{pa 2150 1930}%
\special{fp}%
\special{pa 2410 1610}%
\special{pa 2130 1890}%
\special{fp}%
\special{pa 2390 1570}%
\special{pa 2100 1860}%
\special{fp}%
\special{pa 2360 1540}%
\special{pa 2080 1820}%
\special{fp}%
\special{pa 2330 1510}%
\special{pa 2050 1790}%
\special{fp}%
\special{pa 2300 1480}%
\special{pa 2020 1760}%
\special{fp}%
\special{pa 2270 1450}%
\special{pa 1990 1730}%
\special{fp}%
\special{pa 2240 1420}%
\special{pa 1950 1710}%
\special{fp}%
\special{pa 2200 1400}%
\special{pa 1920 1680}%
\special{fp}%
\special{pa 2170 1370}%
\special{pa 1880 1660}%
\special{fp}%
\special{pa 2130 1350}%
\special{pa 1840 1640}%
\special{fp}%
\special{pa 2090 1330}%
\special{pa 1790 1630}%
\special{fp}%
\special{pa 2050 1310}%
\special{pa 1740 1620}%
\special{fp}%
\special{pa 2010 1290}%
\special{pa 1690 1610}%
\special{fp}%
\special{pa 1970 1270}%
\special{pa 1640 1600}%
\special{fp}%
\special{pa 1920 1260}%
\special{pa 1580 1600}%
\special{fp}%
\special{pa 1880 1240}%
\special{pa 1510 1610}%
\special{fp}%
\special{pa 1830 1230}%
\special{pa 1430 1630}%
\special{fp}%
\special{pa 1780 1220}%
\special{pa 1330 1670}%
\special{fp}%
\special{pa 1420 1220}%
\special{pa 610 2030}%
\special{fp}%
\special{pa 1340 1240}%
\special{pa 630 1950}%
\special{fp}%
\special{pa 1250 1270}%
\special{pa 660 1860}%
\special{fp}%
\special{pa 1140 1320}%
\special{pa 720 1740}%
\special{fp}%
% STR 2 0 3 0 Black White  
% 4 1220 3400 1220 3500 2 0 0 0
% Baumkuchen $B$
\put(12.2000,-35.0000){\makebox(0,0)[lb]{Baumkuchen $B$}}%
% STR 2 0 3 0 Black White  
% 4 3970 2550 3970 2650 2 0 0 0
% $\bullet$
\put(39.7000,-26.5000){\makebox(0,0)[lb]{$\bullet$}}%
% STR 2 0 3 0 Black White  
% 4 4030 2380 4030 2480 2 0 0 0
% $\scriptstyle P$
\put(40.3000,-24.8000){\makebox(0,0)[lb]{$\scriptstyle P$}}%
% LINE 2 0 3 0 Black White  
% 4 4000 3170 4000 1240 4000 1240 4000 1240
% 
\special{pn 8}%
\special{pa 4000 3170}%
\special{pa 4000 1240}%
\special{fp}%
\special{pa 4000 1240}%
\special{pa 4000 1240}%
\special{fp}%
% LINE 2 0 3 0 Black White  
% 2 3300 2610 5100 2610
% 
\special{pn 8}%
\special{pa 3300 2610}%
\special{pa 5100 2610}%
\special{fp}%
% LINE 2 0 3 0 Black White  
% 2 4530 3130 3260 1890
% 
\special{pn 8}%
\special{pa 4530 3130}%
\special{pa 3260 1890}%
\special{fp}%
% LINE 2 0 3 0 Black White  
% 6 3600 2980 5040 1670 5040 1670 5040 1670 5040 1670 5040 1670
% 
\special{pn 8}%
\special{pa 3600 2980}%
\special{pa 5040 1670}%
\special{fp}%
\special{pa 5040 1670}%
\special{pa 5040 1670}%
\special{fp}%
\special{pa 5040 1670}%
\special{pa 5040 1670}%
\special{fp}%
% STR 2 0 3 0 Black White  
% 4 3920 1050 3920 1150 2 0 0 0
% $\scriptstyle A_{2}$
\put(39.2000,-11.5000){\makebox(0,0)[lb]{$\scriptstyle A_{2}$}}%
% STR 2 0 3 0 Black White  
% 4 5100 1540 5100 1640 2 0 0 0
% $\scriptstyle A_{3}$
\put(51.0000,-16.4000){\makebox(0,0)[lb]{$\scriptstyle A_{3}$}}%
% STR 2 0 3 0 Black White  
% 4 5180 2600 5180 2700 2 0 0 0
% $\scriptstyle A_{4}$
\put(51.8000,-27.0000){\makebox(0,0)[lb]{$\scriptstyle A_{4}$}}%
% STR 2 0 3 0 Black White  
% 4 4560 3160 4560 3260 2 0 0 0
% $\scriptstyle A_{5}$
\put(45.6000,-32.6000){\makebox(0,0)[lb]{$\scriptstyle A_{5}$}}%
% STR 2 0 3 0 Black White  
% 4 3420 3030 3420 3130 2 0 0 0
% $\scriptstyle A_{7}$
\put(34.2000,-31.3000){\makebox(0,0)[lb]{$\scriptstyle A_{7}$}}%
% STR 2 0 3 0 Black White  
% 4 3080 2580 3080 2680 2 0 0 0
% $\scriptstyle A_{8}$
\put(30.8000,-26.8000){\makebox(0,0)[lb]{$\scriptstyle A_{8}$}}%
% STR 2 0 3 0 Black White  
% 4 3080 1780 3080 1880 2 0 0 0
% $\scriptstyle A_{1}$
\put(30.8000,-18.8000){\makebox(0,0)[lb]{$\scriptstyle A_{1}$}}%
% STR 2 0 3 0 Black White  
% 4 3940 3210 3940 3310 2 0 0 0
% $\scriptstyle A_{6}$
\put(39.4000,-33.1000){\makebox(0,0)[lb]{$\scriptstyle A_{6}$}}%
% STR 2 0 3 0 Black White  
% 4 3840 1520 3840 1620 2 0 0 0
% $\scriptstyle \tilde{A}_{2}$
\put(38.4000,-16.2000){\makebox(0,0)[lb]{$\scriptstyle \tilde{A}_{2}$}}%
% STR 2 0 3 0 Black White  
% 4 4790 1920 4790 2020 2 0 0 0
% $\scriptstyle \tilde{A}_{3}$
\put(47.9000,-20.2000){\makebox(0,0)[lb]{$\scriptstyle \tilde{A}_{3}$}}%
% STR 2 0 3 0 Black White  
% 4 4760 2460 4760 2560 2 0 0 0
% $\scriptstyle \tilde{A}_{4}$
\put(47.6000,-25.6000){\makebox(0,0)[lb]{$\scriptstyle \tilde{A}_{4}$}}%
% STR 2 0 3 0 Black White  
% 4 4350 2820 4350 2920 2 0 0 0
% $\scriptstyle \tilde{A}_{5}$
\put(43.5000,-29.2000){\makebox(0,0)[lb]{$\scriptstyle \tilde{A}_{5}$}}%
% STR 2 0 3 0 Black White  
% 4 4030 2850 4030 2950 2 0 0 0
% $\scriptstyle \tilde{A}_{6}$
\put(40.3000,-29.5000){\makebox(0,0)[lb]{$\scriptstyle \tilde{A}_{6}$}}%
% STR 2 0 3 0 Black White  
% 4 3660 2680 3660 2780 2 0 0 0
% $\scriptstyle \tilde{A}_{7}$
\put(36.6000,-27.8000){\makebox(0,0)[lb]{$\scriptstyle \tilde{A}_{7}$}}%
% STR 2 0 3 0 Black White  
% 4 3500 2460 3500 2560 2 0 0 0
% $\scriptstyle \tilde{A}_{8}$
\put(35.0000,-25.6000){\makebox(0,0)[lb]{$\scriptstyle \tilde{A}_{8}$}}%
% STR 2 0 3 0 Black White  
% 4 3440 2200 3440 2300 2 0 0 0
% $\scriptstyle \tilde{A}_{1}$
\put(34.4000,-23.0000){\makebox(0,0)[lb]{$\scriptstyle \tilde{A}_{1}$}}%
% STR 2 0 3 0 Black White  
% 4 1390 3700 1390 3800 2 0 0 0
% {\bf Figure 1}
\put(13.9000,-38.0000){\makebox(0,0)[lb]{{\bf Figure 1}}}%
% STR 2 0 3 0 Black White  
% 4 4010 3700 4010 3800 2 0 0 0
% {\bf Figure 2}
\put(40.1000,-38.0000){\makebox(0,0)[lb]{{\bf Figure 2}}}%
% LINE 3 0 3 0 Black White  
% 48 4500 1260 4150 1610 4450 1250 4080 1620 4400 1240 4010 1630 4350 1230 4000 1580 4300 1220 4000 1520 4250 1210 4000 1460 4190 1210 4000 1400 4130 1210 4000 1340 4060 1220 4000 1280 4540 1280 4210 1610 4580 1300 4260 1620 4620 1320 4320 1620 4660 1340 4360 1640 4700 1360 4410 1650 4740 1380 4450 1670 4770 1410 4490 1690 4810 1430 4530 1710 4840 1460 4560 1740 4870 1490 4590 1770 4900 1520 4630 1790 4930 1550 4650 1830 4960 1580 4680 1860 4980 1620 4700 1900 5010 1650 4720 1940
% 
\special{pn 4}%
\special{pa 4500 1260}%
\special{pa 4150 1610}%
\special{fp}%
\special{pa 4450 1250}%
\special{pa 4080 1620}%
\special{fp}%
\special{pa 4400 1240}%
\special{pa 4010 1630}%
\special{fp}%
\special{pa 4350 1230}%
\special{pa 4000 1580}%
\special{fp}%
\special{pa 4300 1220}%
\special{pa 4000 1520}%
\special{fp}%
\special{pa 4250 1210}%
\special{pa 4000 1460}%
\special{fp}%
\special{pa 4190 1210}%
\special{pa 4000 1400}%
\special{fp}%
\special{pa 4130 1210}%
\special{pa 4000 1340}%
\special{fp}%
\special{pa 4060 1220}%
\special{pa 4000 1280}%
\special{fp}%
\special{pa 4540 1280}%
\special{pa 4210 1610}%
\special{fp}%
\special{pa 4580 1300}%
\special{pa 4260 1620}%
\special{fp}%
\special{pa 4620 1320}%
\special{pa 4320 1620}%
\special{fp}%
\special{pa 4660 1340}%
\special{pa 4360 1640}%
\special{fp}%
\special{pa 4700 1360}%
\special{pa 4410 1650}%
\special{fp}%
\special{pa 4740 1380}%
\special{pa 4450 1670}%
\special{fp}%
\special{pa 4770 1410}%
\special{pa 4490 1690}%
\special{fp}%
\special{pa 4810 1430}%
\special{pa 4530 1710}%
\special{fp}%
\special{pa 4840 1460}%
\special{pa 4560 1740}%
\special{fp}%
\special{pa 4870 1490}%
\special{pa 4590 1770}%
\special{fp}%
\special{pa 4900 1520}%
\special{pa 4630 1790}%
\special{fp}%
\special{pa 4930 1550}%
\special{pa 4650 1830}%
\special{fp}%
\special{pa 4960 1580}%
\special{pa 4680 1860}%
\special{fp}%
\special{pa 4980 1620}%
\special{pa 4700 1900}%
\special{fp}%
\special{pa 5010 1650}%
\special{pa 4720 1940}%
\special{fp}%
% LINE 3 0 3 0 Black White  
% 20 4000 2780 3720 3060 4000 2840 3760 3080 4000 2900 3800 3100 4000 2960 3840 3120 4000 3020 3880 3140 4000 3080 3930 3150 4000 3140 3980 3160 3960 2760 3680 3040 3920 2740 3640 3020 3770 2830 3610 2990
% 
\special{pn 4}%
\special{pa 4000 2780}%
\special{pa 3720 3060}%
\special{fp}%
\special{pa 4000 2840}%
\special{pa 3760 3080}%
\special{fp}%
\special{pa 4000 2900}%
\special{pa 3800 3100}%
\special{fp}%
\special{pa 4000 2960}%
\special{pa 3840 3120}%
\special{fp}%
\special{pa 4000 3020}%
\special{pa 3880 3140}%
\special{fp}%
\special{pa 4000 3080}%
\special{pa 3930 3150}%
\special{fp}%
\special{pa 4000 3140}%
\special{pa 3980 3160}%
\special{fp}%
\special{pa 3960 2760}%
\special{pa 3680 3040}%
\special{fp}%
\special{pa 3920 2740}%
\special{pa 3640 3020}%
\special{fp}%
\special{pa 3770 2830}%
\special{pa 3610 2990}%
\special{fp}%
% STR 2 0 3 0 Black White  
% 4 4140 3400 4140 3500 2 0 0 0
% $n=4$
\put(41.4000,-35.0000){\makebox(0,0)[lb]{$n=4$}}%
\end{picture}}%
\]

Throughout this paper, we assume the following.
Let $D$ and $\tilde{D}$ be two disks centered at the point $O$ such that $\tilde{D}$ is a subset of $D$.
Let $B$ be the absolute complement of the interior of $\tilde{D}$ in $D$.
Then its shape represents Baumkuchen (see {\bf Figure 1}).
Let $P$ be a point of $\tilde{D}$,
and let $n$ be an even number of $4$ or more.
Then divide the sets $D, \ \tilde{D}$ and $B$ into $2n$ slices by constructing $n$ straight lines through $P$,
where the lines meet to form $2n$ equal angles.
In the clockwise direction,
let $A_{1},\ A_{2},\ \cdots ,\ \mbox{and}\ A_{2n}$ 
be the intersections of the straight lines passing through the point $P$ and the boundary of $D$.
Similarly,
let $\tilde{A}_{1},\ \tilde{A}_{2},\ \cdots, \ \mbox{and} \ \tilde{A}_{2n}$ be the intersections on the boundary of $\tilde{D}$ (see {\bf Figure 2}).
We denote the areas of the slice $A_{k}A_{k+1}P$ and sector $A_{k}A_{k+1}O$ by $Sl(A_{k})_{P}$ and $Sec(A_{k})_{O}$,
respectively.
In particular,
$Sl(A_{2n})_{P}$ and $Sec(A_{2n})_{O}$ are the areas of the slice $A_{2n}A_{1}P$ and sector $A_{2n} A_{1} O$,
respectively.
Similarly we use $Sl(\tilde{A_{k}})_{P}$ and $Sec(\tilde{A}_{k})_{O}$ on the disk $\tilde{D}$.
Let $Piece(A_{k})$ denote the area of piece $A_{k}A_{k+1}\tilde{A}_{k+1} \tilde{A}_{k}$ on the Baumkuchen $B$.
Then we obtain the following.
\begin{thm}[Baumkuchen Theorem]\label{Baumkuchen Theorem}
For each $1 \leq k \leq n$,
the sum of $Piece(A_{k})$ and $Piece(A_{k+n})$ is $\frac{1}{n}$ times the area of the Baumkuchen $B$.
\end{thm}
\begin{proof}
There is a disk $D^{\prime}$ whose radius is the line segment $OP$.
Let $B_{1}$ and $B_{2}$ be two Baumkuchens,
where $B_{1}=D \setminus {\rm Int}\ D^{\prime}$ and $B_{2}=\tilde{D} \setminus {\rm Int}\ D^{\prime}$.
Applying Lemma \ref{lemma3} on $B_{1}$ and $B_{2}$,
we deduce the Baumkuchen Theorem,
as claimed.
\end{proof}
To prove this,
we first present Lemmas \ref{lemma2} and \ref{lemma3}.
Suppose the point $P$ is on the boundary of the disk $\tilde{D}$ (see {\bf Figure 3}).
Then there are intersections $P,\ \tilde{A}_{1},\ \tilde{A}_{2} \cdots ,\ \mbox{and} \ \tilde{A}_{n}$ between $n$ straight lines passing through $P$ and the boundary of the disk $\tilde{D}$,
where number the intersections clockwise from point $P$ in succession.
\[
%WinTpicVersion4.32a
{\unitlength 0.1in%
\begin{picture}(21.6000,27.1000)(12.6000,-39.3000)%
% CIRCLE 2 0 3 0 Black White  
% 4 2400 2420 2400 1820 2400 1820 2400 1820
% 
\special{pn 8}%
\special{ar 2400 2420 600 600 0.0000000 6.2831853}%
% CIRCLE 2 0 3 0 Black White  
% 4 2400 2420 2400 1400 2400 1400 2400 1400
% 
\special{pn 8}%
\special{ar 2400 2420 1020 1020 0.0000000 6.2831853}%
% LINE 2 0 3 0 Black White  
% 2 2200 3420 2200 1420
% 
\special{pn 8}%
\special{pa 2200 3420}%
\special{pa 2200 1420}%
\special{fp}%
% LINE 2 0 3 0 Black White  
% 2 1550 2980 3260 2980
% 
\special{pn 8}%
\special{pa 1550 2980}%
\special{pa 3260 2980}%
\special{fp}%
% LINE 2 0 3 0 Black White  
% 2 1880 3300 3280 1900
% 
\special{pn 8}%
\special{pa 1880 3300}%
\special{pa 3280 1900}%
\special{fp}%
% LINE 2 0 3 0 Black White  
% 2 2600 3410 1400 2160
% 
\special{pn 8}%
\special{pa 2600 3410}%
\special{pa 1400 2160}%
\special{fp}%
% STR 2 0 3 0 Black White  
% 4 2370 2360 2370 2460 2 0 0 0
% $\bullet$
\put(23.7000,-24.6000){\makebox(0,0)[lb]{$\bullet$}}%
% STR 2 0 3 0 Black White  
% 4 2170 2920 2170 3020 2 0 0 0
% $\bullet$
\put(21.7000,-30.2000){\makebox(0,0)[lb]{$\bullet$}}%
% STR 2 0 3 0 Black White  
% 4 2140 1250 2140 1350 2 0 0 0
% $\scriptstyle A_{2}$
\put(21.4000,-13.5000){\makebox(0,0)[lb]{$\scriptstyle A_{2}$}}%
% STR 2 0 3 0 Black White  
% 4 1260 2030 1260 2130 2 0 0 0
% $\scriptstyle A_{1}$
\put(12.6000,-21.3000){\makebox(0,0)[lb]{$\scriptstyle A_{1}$}}%
% STR 2 0 3 0 Black White  
% 4 3360 1770 3360 1870 2 0 0 0
% $\scriptstyle A_{3}$
\put(33.6000,-18.7000){\makebox(0,0)[lb]{$\scriptstyle A_{3}$}}%
% STR 2 0 3 0 Black White  
% 4 3370 2950 3370 3050 2 0 0 0
% $\scriptstyle A_{4}$
\put(33.7000,-30.5000){\makebox(0,0)[lb]{$\scriptstyle A_{4}$}}%
% STR 2 0 3 0 Black White  
% 4 2610 3450 2610 3550 2 0 0 0
% $\scriptstyle A_{5}$
\put(26.1000,-35.5000){\makebox(0,0)[lb]{$\scriptstyle A_{5}$}}%
% STR 2 0 3 0 Black White  
% 4 2150 3500 2150 3600 2 0 0 0
% $\scriptstyle A_{6}$
\put(21.5000,-36.0000){\makebox(0,0)[lb]{$\scriptstyle A_{6}$}}%
% STR 2 0 3 0 Black White  
% 4 1720 3350 1720 3450 2 0 0 0
% $\scriptstyle A_{7}$
\put(17.2000,-34.5000){\makebox(0,0)[lb]{$\scriptstyle A_{7}$}}%
% STR 2 0 3 0 Black White  
% 4 1370 2970 1370 3070 2 0 0 0
% $\scriptstyle A_{8}$
\put(13.7000,-30.7000){\makebox(0,0)[lb]{$\scriptstyle A_{8}$}}%
% STR 2 0 3 0 Black White  
% 4 2390 2180 2390 2280 2 0 0 0
% $\scriptstyle O$
\put(23.9000,-22.8000){\makebox(0,0)[lb]{$\scriptstyle O$}}%
% STR 2 0 3 0 Black White  
% 4 2240 2700 2240 2800 2 0 0 0
% $\scriptstyle P$
\put(22.4000,-28.0000){\makebox(0,0)[lb]{$\scriptstyle P$}}%
% STR 2 0 3 0 Black White  
% 4 1670 2640 1670 2740 2 0 0 0
% $\scriptstyle \tilde{A}_{1}$
\put(16.7000,-27.4000){\makebox(0,0)[lb]{$\scriptstyle \tilde{A}_{1}$}}%
% STR 2 0 3 0 Black White  
% 4 2020 1730 2020 1830 2 0 0 0
% $\scriptstyle \tilde{A}_{2}$
\put(20.2000,-18.3000){\makebox(0,0)[lb]{$\scriptstyle \tilde{A}_{2}$}}%
% STR 2 0 3 0 Black White  
% 4 3070 2200 3070 2300 2 0 0 0
% $\scriptstyle \tilde{A}_{3}$
\put(30.7000,-23.0000){\makebox(0,0)[lb]{$\scriptstyle \tilde{A}_{3}$}}%
% STR 2 0 3 0 Black White  
% 4 2590 3050 2590 3150 2 0 0 0
% $\scriptstyle \tilde{A}_{4}$
\put(25.9000,-31.5000){\makebox(0,0)[lb]{$\scriptstyle \tilde{A}_{4}$}}%
% LINE 3 0 3 0 Black White  
% 38 2200 2180 1820 2560 2200 2240 1830 2610 2200 2300 1860 2640 2200 2360 1890 2670 2200 2420 1920 2700 2200 2480 1950 2730 2200 2540 1980 2760 2200 2600 2010 2790 2200 2660 2040 2820 2200 2720 2070 2850 2200 2780 2100 2880 2200 2840 2130 2910 2200 2900 2160 2940 2200 2120 1810 2510 2200 2060 1800 2460 2200 2000 1800 2400 2200 1940 1810 2330 2200 1880 1830 2250 2130 1890 1870 2150
% 
\special{pn 4}%
\special{pa 2200 2180}%
\special{pa 1820 2560}%
\special{fp}%
\special{pa 2200 2240}%
\special{pa 1830 2610}%
\special{fp}%
\special{pa 2200 2300}%
\special{pa 1860 2640}%
\special{fp}%
\special{pa 2200 2360}%
\special{pa 1890 2670}%
\special{fp}%
\special{pa 2200 2420}%
\special{pa 1920 2700}%
\special{fp}%
\special{pa 2200 2480}%
\special{pa 1950 2730}%
\special{fp}%
\special{pa 2200 2540}%
\special{pa 1980 2760}%
\special{fp}%
\special{pa 2200 2600}%
\special{pa 2010 2790}%
\special{fp}%
\special{pa 2200 2660}%
\special{pa 2040 2820}%
\special{fp}%
\special{pa 2200 2720}%
\special{pa 2070 2850}%
\special{fp}%
\special{pa 2200 2780}%
\special{pa 2100 2880}%
\special{fp}%
\special{pa 2200 2840}%
\special{pa 2130 2910}%
\special{fp}%
\special{pa 2200 2900}%
\special{pa 2160 2940}%
\special{fp}%
\special{pa 2200 2120}%
\special{pa 1810 2510}%
\special{fp}%
\special{pa 2200 2060}%
\special{pa 1800 2460}%
\special{fp}%
\special{pa 2200 2000}%
\special{pa 1800 2400}%
\special{fp}%
\special{pa 2200 1940}%
\special{pa 1810 2330}%
\special{fp}%
\special{pa 2200 1880}%
\special{pa 1830 2250}%
\special{fp}%
\special{pa 2130 1890}%
\special{pa 1870 2150}%
\special{fp}%
% LINE 3 0 3 0 Black White  
% 16 3000 2400 2420 2980 2990 2350 2360 2980 2990 2290 2300 2980 2970 2250 2240 2980 3000 2460 2480 2980 2990 2530 2540 2980 2960 2620 2600 2980 2910 2730 2710 2930
% 
\special{pn 4}%
\special{pa 3000 2400}%
\special{pa 2420 2980}%
\special{fp}%
\special{pa 2990 2350}%
\special{pa 2360 2980}%
\special{fp}%
\special{pa 2990 2290}%
\special{pa 2300 2980}%
\special{fp}%
\special{pa 2970 2250}%
\special{pa 2240 2980}%
\special{fp}%
\special{pa 3000 2460}%
\special{pa 2480 2980}%
\special{fp}%
\special{pa 2990 2530}%
\special{pa 2540 2980}%
\special{fp}%
\special{pa 2960 2620}%
\special{pa 2600 2980}%
\special{fp}%
\special{pa 2910 2730}%
\special{pa 2710 2930}%
\special{fp}%
% STR 2 0 3 0 Black White  
% 4 2240 3710 2240 3810 2 0 0 0
% $n=4$
\put(22.4000,-38.1000){\makebox(0,0)[lb]{$n=4$}}%
% STR 2 0 3 0 Black White  
% 4 2140 3960 2140 4060 2 0 0 0
% {\bf Figure 3}
\put(21.4000,-40.6000){\makebox(0,0)[lb]{{\bf Figure 3}}}%
% LINE 0 0 3 0 Black White  
% 4 2190 1860 1840 2640 1840 2640 1840 2640
% 
\special{pn 20}%
\special{pa 2190 1860}%
\special{pa 1840 2640}%
\special{fp}%
\special{pa 1840 2640}%
\special{pa 1840 2640}%
\special{fp}%
% LINE 0 0 3 0 Black White  
% 2 2210 1860 2950 2210
% 
\special{pn 20}%
\special{pa 2210 1860}%
\special{pa 2950 2210}%
\special{fp}%
% LINE 0 0 3 0 Black White  
% 2 2960 2220 2620 2970
% 
\special{pn 20}%
\special{pa 2960 2220}%
\special{pa 2620 2970}%
\special{fp}%
% LINE 0 0 3 0 Black White  
% 2 2630 2980 1850 2630
% 
\special{pn 20}%
\special{pa 2630 2980}%
\special{pa 1850 2630}%
\special{fp}%
% LINE 1 0 3 0 Black White  
% 2 2620 2970 2190 1860
% 
\special{pn 13}%
\special{pa 2620 2970}%
\special{pa 2190 1860}%
\special{fp}%
% LINE 1 0 3 0 Black White  
% 2 1850 2630 2970 2220
% 
\special{pn 13}%
\special{pa 1850 2630}%
\special{pa 2970 2220}%
\special{fp}%
\end{picture}}%
\]
\begin{lmm}\label{lemma2}
Suppose the point $P$ is on the boundary of the disk $\tilde{D}$.
Then we have the followings.
\begin{enumerate}
\item
%Polyhedron $\tilde{A}_{1}\tilde{A}_{2} \cdots \tilde{A}_{n}$ is a regular n-sided polygon.
The $n$-sided polygon $\tilde{A}_{1}\tilde{A}_{2} \cdots \tilde{A}_{n}$ is regular.
\item
For each $ 1 \leq k \leq \frac{n}{2}-1$,
the sum of $Sl(\tilde{A}_{k})_{P}$ and $Sl(\tilde{A}_{k+\frac{n}{2}})_{P}$ is $\frac{2}{n}$ times the area of the disk $\tilde{D}$.
\item
%Let $\tilde{E}_{k}$ be the area of the figure surrounded by the line segment $\tilde{A}_{k}P$ and the arc $\tilde{A}_{k}P$,
%where $k$ is either $1$ or $n$.
%Then the sum of $Sl(\tilde{A}_{\frac{n}{2}})_{P},\ \tilde{E}_{1}$ and $\tilde{E}_{k}$ is $\frac{2}{n}$ times the area of the disk $\tilde{D}$.
The sum of $Sl(\tilde{A}_{\frac{n}{2}})_{P}$ and 
the area of the figure surrounded by the line segments $\tilde{A}_{1}P,\ \tilde{A}_{n}P$ and arc $\tilde{A}_{1}P\tilde{A}_{n}$
is $\frac{2}{n}$ times the area of the disk $\tilde{D}$.
\end{enumerate}
\end{lmm}
\begin{proof}
According to the inscribed angle theorem,
for each $k=1,2 \cdots  n-1$,
the angle $\tilde{A}_{k}\tilde{A}_{k+1}O$ is twice the angle $\tilde{A}_{k} \tilde{A}_{k+1}P$.
Thus,
the $n$-sided polygon $\tilde{A}_{1} \tilde{A}_{2} \cdots \tilde{A}_{n}$ is regular.
Since all opposite sides of the regular $n$ polygon are parallel and of equal length,
%Since all subteness of the regular $n$ polygon are parallel and equal in length,
the sum of the areas of the triangles $\tilde{A}_{k} \tilde{A}_{k+1}P$ and $\tilde{A}_{k+\frac{n}{2}} \tilde{A}_{k+\frac{n}{2}+1}P$ is equal to the sum of the areas of the triangles 
 $\tilde{A}_{k} \tilde{A}_{k+1}O$ and $\tilde{A}_{k+\frac{n}{2}} \tilde{A}_{k+\frac{n}{2}+1}O$,
where $k=1,2,\cdots , \frac{n}{2}-1$ (see {\bf Figure 3}).
Thus the sum of $Sl(\tilde{A}_{k})_{P}$ and $Sl(\tilde{A}_{k+\frac{n}{2}})_{P}$ is $\frac{2}{n}$ times the area of the disk $\tilde{D}$.
Using condition $(2)$,
we can easily deduce condition $(3)$.
%It is not difficult to prove the condition (3) by using the condition (2).
\end{proof}
\begin{lmm}\label{lemma3}
Suppose the point $P$ is on the boundary of the disk $\tilde{D}$.
For each $1 \leq k \leq n$,
the sum of $Piece(A_{k})$ and $Piece(A_{k+n})$ is $\frac{1}{n}$ times the area of the Baumkuchen $B$.
\end{lmm}
\begin{proof}
%For any $k=1,2,\cdots , n-1$,
%we shall show that 
%the sum of areas of the piece $A_{k}A_{k+1}B_{k+1}B_{k}$ and the slice $B_{k+n}B_{k+n+1}P$ 
%is one-$n$th of the area of $E$ (see {\bf Figure 2}).
Due to the symmetry,
the lengths of the line segments $A_{k}\tilde{A}_{k}$ and $A_{k+n}P$ are equal.
Divide $Piece(A_{k})$
by a line segment $C_{1}\tilde{C}_{1}$ connecting the point $C_{1}$ on the arc $A_{k}A_{k+1}$ 
and point $\tilde{C}_{1}$ on the arc $\tilde{A}_{k}\tilde{A}_{k+1}$.
We obtain a piece $F_{1}G_{1}\tilde{G}_{1}\tilde{F}_{1}$ (see {\bf Figure 4})
constructed by pasting the pieces $A_{k}C_{1}\tilde{C}_{1}\tilde{A}_{k}$ and 
$A_{k+1}C_{1}\tilde{C}_{1}\tilde{A}_{k+1}$ along with the line segments $A_{k+n}P$ and $A_{k+n+1}P$,
respectively.
Then $\tilde{F}_{1}$ and $\tilde{G}_{1}$ on the boundary of the disk $\tilde{D}$ are points where the point $\tilde{C}_{1}$ is split into two points.
Similarly,
$F_{1}$ and $G_{1}$ are two points on the boundary of the disk $D$.
It is clear that the sum of $Piece (A_{k})$ and $Piece (A_{k+n})$ is equal to the area of the piece $F_{1}G_{1}\tilde{G}_{1}\tilde{F}_{1}$.
%Then the points that split $C_{1}$ and $C_{1}^{\prime}$ are $F_{1},\ F^{\prime}_{1}$ and $G_{1},\ G^{\prime}_{1}$,
%respectively.
Using the symmetry,
we inductively obtain a piece $F_{l}G_{l}\tilde{G}_{l}\tilde{F}_{l}$ (see {\bf Figure 5}) constructed by pasting the pieces 
$G_{l-1}C_{l}\tilde{C}_{l}\tilde{G}_{l-1}$ and $F_{l-1}C_{l}\tilde{C}_{l}\tilde{F}_{l-1}$ along the line segment 
$C_{l-1}\tilde{C}_{l-1}$,
where $C_{l}$ and $\tilde{C}_{l}$ are points on the arc $F_{l-1}G_{l-1}$ in the disk $D$ and on the arc $\tilde{F}_{l-1}\tilde{G}_{l-1}$ in the disk $\tilde{D}$,
respectively.
There is a straight line $C_{l}\tilde{C}_{l}$ passing through the point $O$.
%By the condition (1) of Lemma \ref{lemma2},
%Then the area of the piece $F_{l}G_{l}\tilde{G}_{l}\tilde{F}_{l}$ is one-$n$th of the area of the baumkuchen $B$ 
Then the length of arc $\tilde{F}_{l}\tilde{G}_{l}$ is $\frac{1}{n}$ of the circumference of the disk $\tilde{D}$
by the condition (1) of Lemma \ref{lemma2}.
%so we get a piece $A_{k}A_{k+1}B_{k+1}B_{n+k+1}$ into $E$ by sticking slice $B_{k+n}B_{k+n+1}P$ along side $ A_{k}B_{k}$.
%Then the point $P$ overlaps with the point $A_{k}$.
%Divide this piece $A_{k}A_{k+1}B_{k+1}B_{n+k+1}$ 
%by a line segment $FF^{\prime}$ connecting the point $F$ of arc $A_{k}A_{k+1}$ 
%and the point $F^{\prime}$ of arc $B_{k+1}B_{n+k+1}$.
Since the area of the piece $F_{l}G_{l}\tilde{G}_{l}\tilde{F}_{l}$ is $\frac{1}{n}$ times the area of the Baumkuchen $B$,
so is the sum of $Piece (A_{k})$ and $Piece (A_{k+n})$.
\end{proof}
\[
%WinTpicVersion4.32a
{\unitlength 0.1in%
\begin{picture}(48.2000,24.4000)(8.4000,-40.7000)%
% CIRCLE 2 0 3 0 Black White  
% 4 2000 2800 2000 2190 2000 2190 2000 2190
% 
\special{pn 8}%
\special{ar 2000 2800 610 610 0.0000000 6.2831853}%
% CIRCLE 2 0 3 0 Black White  
% 4 2000 2800 2000 1810 2000 1810 2000 1810
% 
\special{pn 8}%
\special{ar 2000 2800 990 990 0.0000000 6.2831853}%
% CIRCLE 2 0 3 0 Black White  
% 4 4600 2820 4600 1830 4600 1830 4600 1830
% 
\special{pn 8}%
\special{ar 4600 2820 990 990 0.0000000 6.2831853}%
% CIRCLE 2 0 3 0 Black White  
% 4 4600 2820 4600 2210 4600 2210 4600 2210
% 
\special{pn 8}%
\special{ar 4600 2820 610 610 0.0000000 6.2831853}%
% LINE 2 0 3 0 Black White  
% 2 1800 3380 1800 3780
% 
\special{pn 8}%
\special{pa 1800 3380}%
\special{pa 1800 3780}%
\special{fp}%
% LINE 2 1 3 0 Black White  
% 2 1800 3380 1800 2230
% 
\special{pn 8}%
\special{pa 1800 3380}%
\special{pa 1800 2230}%
\special{da 0.070}%
% LINE 2 0 3 0 Black White  
% 2 1800 2230 1800 1840
% 
\special{pn 8}%
\special{pa 1800 2230}%
\special{pa 1800 1840}%
\special{fp}%
% LINE 2 1 3 0 Black White  
% 2 1480 3640 2920 2410
% 
\special{pn 8}%
\special{pa 1480 3640}%
\special{pa 2920 2410}%
\special{da 0.070}%
% LINE 2 0 3 0 Black White  
% 2 2920 2410 2590 2690
% 
\special{pn 8}%
\special{pa 2920 2410}%
\special{pa 2590 2690}%
\special{fp}%
% LINE 2 0 3 0 Black White  
% 2 1470 3650 1800 3370
% 
\special{pn 8}%
\special{pa 1470 3650}%
\special{pa 1800 3370}%
\special{fp}%
% STR 2 0 3 0 Black White  
% 4 1950 2740 1950 2840 2 0 0 0
% $\bullet$
\put(19.5000,-28.4000){\makebox(0,0)[lb]{$\bullet$}}%
% LINE 2 1 3 0 Black White  
% 2 2380 3290 1600 2340
% 
\special{pn 8}%
\special{pa 2380 3290}%
\special{pa 1600 2340}%
\special{da 0.070}%
% LINE 2 1 3 0 Black White  
% 2 1400 2810 2610 2800
% 
\special{pn 8}%
\special{pa 1400 2810}%
\special{pa 2610 2800}%
\special{da 0.070}%
% LINE 0 0 3 0 Black White  
% 2 1040 3000 1430 3000
% 
\special{pn 20}%
\special{pa 1040 3000}%
\special{pa 1430 3000}%
\special{fp}%
% LINE 0 0 3 0 Black White  
% 2 2200 2230 2200 1840
% 
\special{pn 20}%
\special{pa 2200 2230}%
\special{pa 2200 1840}%
\special{fp}%
% LINE 0 0 3 0 Black White  
% 2 2200 3760 2200 3380
% 
\special{pn 20}%
\special{pa 2200 3760}%
\special{pa 2200 3380}%
\special{fp}%
% LINE 3 0 3 0 Black White  
% 26 2180 1840 1810 2210 2140 1820 1800 2160 2090 1810 1800 2100 2030 1810 1800 2040 1970 1810 1800 1980 1900 1820 1800 1920 1830 1830 1800 1860 2190 1890 1880 2200 2190 1950 1950 2190 2190 2010 2010 2190 2190 2070 2060 2200 2190 2130 2120 2200 2190 2190 2170 2210
% 
\special{pn 4}%
\special{pa 2180 1840}%
\special{pa 1810 2210}%
\special{fp}%
\special{pa 2140 1820}%
\special{pa 1800 2160}%
\special{fp}%
\special{pa 2090 1810}%
\special{pa 1800 2100}%
\special{fp}%
\special{pa 2030 1810}%
\special{pa 1800 2040}%
\special{fp}%
\special{pa 1970 1810}%
\special{pa 1800 1980}%
\special{fp}%
\special{pa 1900 1820}%
\special{pa 1800 1920}%
\special{fp}%
\special{pa 1830 1830}%
\special{pa 1800 1860}%
\special{fp}%
\special{pa 2190 1890}%
\special{pa 1880 2200}%
\special{fp}%
\special{pa 2190 1950}%
\special{pa 1950 2190}%
\special{fp}%
\special{pa 2190 2010}%
\special{pa 2010 2190}%
\special{fp}%
\special{pa 2190 2070}%
\special{pa 2060 2200}%
\special{fp}%
\special{pa 2190 2130}%
\special{pa 2120 2200}%
\special{fp}%
\special{pa 2190 2190}%
\special{pa 2170 2210}%
\special{fp}%
% LINE 3 0 3 0 Black White  
% 24 2120 3400 1800 3720 2190 3390 1810 3770 2190 3450 1860 3780 2190 3510 1910 3790 2190 3570 1970 3790 2190 3630 2030 3790 2190 3690 2090 3790 2060 3400 1800 3660 1990 3410 1800 3600 1940 3400 1800 3540 1880 3400 1800 3480 1830 3390 1800 3420
% 
\special{pn 4}%
\special{pa 2120 3400}%
\special{pa 1800 3720}%
\special{fp}%
\special{pa 2190 3390}%
\special{pa 1810 3770}%
\special{fp}%
\special{pa 2190 3450}%
\special{pa 1860 3780}%
\special{fp}%
\special{pa 2190 3510}%
\special{pa 1910 3790}%
\special{fp}%
\special{pa 2190 3570}%
\special{pa 1970 3790}%
\special{fp}%
\special{pa 2190 3630}%
\special{pa 2030 3790}%
\special{fp}%
\special{pa 2190 3690}%
\special{pa 2090 3790}%
\special{fp}%
\special{pa 2060 3400}%
\special{pa 1800 3660}%
\special{fp}%
\special{pa 1990 3410}%
\special{pa 1800 3600}%
\special{fp}%
\special{pa 1940 3400}%
\special{pa 1800 3540}%
\special{fp}%
\special{pa 1880 3400}%
\special{pa 1800 3480}%
\special{fp}%
\special{pa 1830 3390}%
\special{pa 1800 3420}%
\special{fp}%
% LINE 3 0 3 0 Black White  
% 22 2630 2050 2370 2310 2700 2100 2430 2370 2750 2170 2490 2430 2810 2230 2530 2510 2850 2310 2570 2590 2890 2390 2620 2660 2570 1990 2290 2270 2490 1950 2220 2220 2410 1910 2210 2110 2330 1870 2210 1990 2240 1840 2210 1870
% 
\special{pn 4}%
\special{pa 2630 2050}%
\special{pa 2370 2310}%
\special{fp}%
\special{pa 2700 2100}%
\special{pa 2430 2370}%
\special{fp}%
\special{pa 2750 2170}%
\special{pa 2490 2430}%
\special{fp}%
\special{pa 2810 2230}%
\special{pa 2530 2510}%
\special{fp}%
\special{pa 2850 2310}%
\special{pa 2570 2590}%
\special{fp}%
\special{pa 2890 2390}%
\special{pa 2620 2660}%
\special{fp}%
\special{pa 2570 1990}%
\special{pa 2290 2270}%
\special{fp}%
\special{pa 2490 1950}%
\special{pa 2220 2220}%
\special{fp}%
\special{pa 2410 1910}%
\special{pa 2210 2110}%
\special{fp}%
\special{pa 2330 1870}%
\special{pa 2210 1990}%
\special{fp}%
\special{pa 2240 1840}%
\special{pa 2210 1870}%
\special{fp}%
% LINE 3 0 3 0 Black White  
% 16 1570 3230 1300 3500 1630 3290 1370 3550 1710 3330 1430 3610 1510 3170 1250 3430 1470 3090 1190 3370 1420 3020 1150 3290 1310 3010 1110 3210 1190 3010 1070 3130
% 
\special{pn 4}%
\special{pa 1570 3230}%
\special{pa 1300 3500}%
\special{fp}%
\special{pa 1630 3290}%
\special{pa 1370 3550}%
\special{fp}%
\special{pa 1710 3330}%
\special{pa 1430 3610}%
\special{fp}%
\special{pa 1510 3170}%
\special{pa 1250 3430}%
\special{fp}%
\special{pa 1470 3090}%
\special{pa 1190 3370}%
\special{fp}%
\special{pa 1420 3020}%
\special{pa 1150 3290}%
\special{fp}%
\special{pa 1310 3010}%
\special{pa 1110 3210}%
\special{fp}%
\special{pa 1190 3010}%
\special{pa 1070 3130}%
\special{fp}%
% LINE 0 0 3 0 Black White  
% 14 1800 3540 1630 3710 1800 3600 1670 3730 1800 3660 1720 3740 1800 3720 1760 3760 1800 3480 1590 3690 1800 3420 1550 3670 1710 3450 1510 3650
% 
\special{pn 20}%
\special{pa 1800 3540}%
\special{pa 1630 3710}%
\special{fp}%
\special{pa 1800 3600}%
\special{pa 1670 3730}%
\special{fp}%
\special{pa 1800 3660}%
\special{pa 1720 3740}%
\special{fp}%
\special{pa 1800 3720}%
\special{pa 1760 3760}%
\special{fp}%
\special{pa 1800 3480}%
\special{pa 1590 3690}%
\special{fp}%
\special{pa 1800 3420}%
\special{pa 1550 3670}%
\special{fp}%
\special{pa 1710 3450}%
\special{pa 1510 3650}%
\special{fp}%
% LINE 0 0 3 0 Black White  
% 2 3630 3010 4030 3010
% 
\special{pn 20}%
\special{pa 3630 3010}%
\special{pa 4030 3010}%
\special{fp}%
% LINE 0 0 3 0 Black White  
% 2 4600 3820 4600 3430
% 
\special{pn 20}%
\special{pa 4600 3820}%
\special{pa 4600 3430}%
\special{fp}%
% LINE 0 0 3 0 Black White  
% 2 4600 2200 4600 1830
% 
\special{pn 20}%
\special{pa 4600 2200}%
\special{pa 4600 1830}%
\special{fp}%
% LINE 0 0 3 0 Black White  
% 2 4810 3780 4810 3390
% 
\special{pn 20}%
\special{pa 4810 3780}%
\special{pa 4810 3390}%
\special{fp}%
% LINE 2 1 3 0 Black White  
% 2 4810 3390 4810 2250
% 
\special{pn 8}%
\special{pa 4810 3390}%
\special{pa 4810 2250}%
\special{da 0.070}%
% LINE 0 0 3 0 Black White  
% 2 4810 2250 4810 1860
% 
\special{pn 20}%
\special{pa 4810 2250}%
\special{pa 4810 1860}%
\special{fp}%
% LINE 2 1 3 0 Black White  
% 2 4600 3420 4600 2210
% 
\special{pn 8}%
\special{pa 4600 3420}%
\special{pa 4600 2210}%
\special{da 0.070}%
% LINE 0 0 3 0 Black White  
% 2 5600 2810 5210 2810
% 
\special{pn 20}%
\special{pa 5600 2810}%
\special{pa 5210 2810}%
\special{fp}%
% LINE 2 1 3 0 Black White  
% 2 5210 2810 4010 2810
% 
\special{pn 8}%
\special{pa 5210 2810}%
\special{pa 4010 2810}%
\special{da 0.070}%
% LINE 3 0 3 0 Black White  
% 18 4800 1980 4610 2170 4800 2040 4630 2210 4800 2100 4680 2220 4800 2160 4740 2220 4800 1920 4610 2110 4800 1860 4610 2050 4760 1840 4610 1990 4700 1840 4610 1930 4650 1830 4610 1870
% 
\special{pn 4}%
\special{pa 4800 1980}%
\special{pa 4610 2170}%
\special{fp}%
\special{pa 4800 2040}%
\special{pa 4630 2210}%
\special{fp}%
\special{pa 4800 2100}%
\special{pa 4680 2220}%
\special{fp}%
\special{pa 4800 2160}%
\special{pa 4740 2220}%
\special{fp}%
\special{pa 4800 1920}%
\special{pa 4610 2110}%
\special{fp}%
\special{pa 4800 1860}%
\special{pa 4610 2050}%
\special{fp}%
\special{pa 4760 1840}%
\special{pa 4610 1990}%
\special{fp}%
\special{pa 4700 1840}%
\special{pa 4610 1930}%
\special{fp}%
\special{pa 4650 1830}%
\special{pa 4610 1870}%
\special{fp}%
% LINE 3 0 3 0 Black White  
% 16 4800 3540 4610 3730 4800 3600 4610 3790 4800 3660 4650 3810 4800 3720 4720 3800 4800 3480 4610 3670 4800 3420 4610 3610 4750 3410 4610 3550 4680 3420 4610 3490
% 
\special{pn 4}%
\special{pa 4800 3540}%
\special{pa 4610 3730}%
\special{fp}%
\special{pa 4800 3600}%
\special{pa 4610 3790}%
\special{fp}%
\special{pa 4800 3660}%
\special{pa 4650 3810}%
\special{fp}%
\special{pa 4800 3720}%
\special{pa 4720 3800}%
\special{fp}%
\special{pa 4800 3480}%
\special{pa 4610 3670}%
\special{fp}%
\special{pa 4800 3420}%
\special{pa 4610 3610}%
\special{fp}%
\special{pa 4750 3410}%
\special{pa 4610 3550}%
\special{fp}%
\special{pa 4680 3420}%
\special{pa 4610 3490}%
\special{fp}%
% LINE 1 0 3 0 Black White  
% 56 5360 2200 5100 2460 5390 2230 5120 2500 5420 2260 5140 2540 5440 2300 5160 2580 5460 2340 5180 2620 5480 2380 5190 2670 5500 2420 5200 2720 5520 2460 5210 2770 5530 2510 5240 2800 5550 2550 5300 2800 5560 2600 5360 2800 5570 2650 5420 2800 5580 2700 5480 2800 5590 2750 5540 2800 5330 2170 5070 2430 5310 2130 5040 2400 5280 2100 5010 2370 5240 2080 4980 2340 5210 2050 4950 2310 5180 2020 4910 2290 5140 2000 4870 2270 5100 1980 4830 2250 5070 1950 4820 2200 5030 1930 4820 2140 4990 1910 4820 2080 4940 1900 4820 2020 4900 1880 4820 1960 4850 1870 4820 1900
% 
\special{pn 13}%
\special{pa 5360 2200}%
\special{pa 5100 2460}%
\special{fp}%
\special{pa 5390 2230}%
\special{pa 5120 2500}%
\special{fp}%
\special{pa 5420 2260}%
\special{pa 5140 2540}%
\special{fp}%
\special{pa 5440 2300}%
\special{pa 5160 2580}%
\special{fp}%
\special{pa 5460 2340}%
\special{pa 5180 2620}%
\special{fp}%
\special{pa 5480 2380}%
\special{pa 5190 2670}%
\special{fp}%
\special{pa 5500 2420}%
\special{pa 5200 2720}%
\special{fp}%
\special{pa 5520 2460}%
\special{pa 5210 2770}%
\special{fp}%
\special{pa 5530 2510}%
\special{pa 5240 2800}%
\special{fp}%
\special{pa 5550 2550}%
\special{pa 5300 2800}%
\special{fp}%
\special{pa 5560 2600}%
\special{pa 5360 2800}%
\special{fp}%
\special{pa 5570 2650}%
\special{pa 5420 2800}%
\special{fp}%
\special{pa 5580 2700}%
\special{pa 5480 2800}%
\special{fp}%
\special{pa 5590 2750}%
\special{pa 5540 2800}%
\special{fp}%
\special{pa 5330 2170}%
\special{pa 5070 2430}%
\special{fp}%
\special{pa 5310 2130}%
\special{pa 5040 2400}%
\special{fp}%
\special{pa 5280 2100}%
\special{pa 5010 2370}%
\special{fp}%
\special{pa 5240 2080}%
\special{pa 4980 2340}%
\special{fp}%
\special{pa 5210 2050}%
\special{pa 4950 2310}%
\special{fp}%
\special{pa 5180 2020}%
\special{pa 4910 2290}%
\special{fp}%
\special{pa 5140 2000}%
\special{pa 4870 2270}%
\special{fp}%
\special{pa 5100 1980}%
\special{pa 4830 2250}%
\special{fp}%
\special{pa 5070 1950}%
\special{pa 4820 2200}%
\special{fp}%
\special{pa 5030 1930}%
\special{pa 4820 2140}%
\special{fp}%
\special{pa 4990 1910}%
\special{pa 4820 2080}%
\special{fp}%
\special{pa 4940 1900}%
\special{pa 4820 2020}%
\special{fp}%
\special{pa 4900 1880}%
\special{pa 4820 1960}%
\special{fp}%
\special{pa 4850 1870}%
\special{pa 4820 1900}%
\special{fp}%
% LINE 1 0 3 0 Black White  
% 56 4240 3320 3980 3580 4280 3340 4010 3610 4320 3360 4040 3640 4360 3380 4080 3660 4400 3400 4120 3680 4450 3410 4160 3700 4500 3420 4200 3720 4550 3430 4240 3740 4590 3450 4290 3750 4590 3510 4330 3770 4590 3570 4380 3780 4590 3630 4430 3790 4590 3690 4480 3800 4590 3750 4530 3810 4210 3290 3950 3550 4180 3260 3910 3530 4150 3230 3880 3500 4120 3200 3860 3460 4090 3170 3830 3430 4070 3130 3800 3400 4050 3090 3780 3360 4030 3050 3760 3320 4000 3020 3730 3290 3940 3020 3710 3250 3880 3020 3690 3210 3820 3020 3680 3160 3760 3020 3660 3120 3700 3020 3650 3070
% 
\special{pn 13}%
\special{pa 4240 3320}%
\special{pa 3980 3580}%
\special{fp}%
\special{pa 4280 3340}%
\special{pa 4010 3610}%
\special{fp}%
\special{pa 4320 3360}%
\special{pa 4040 3640}%
\special{fp}%
\special{pa 4360 3380}%
\special{pa 4080 3660}%
\special{fp}%
\special{pa 4400 3400}%
\special{pa 4120 3680}%
\special{fp}%
\special{pa 4450 3410}%
\special{pa 4160 3700}%
\special{fp}%
\special{pa 4500 3420}%
\special{pa 4200 3720}%
\special{fp}%
\special{pa 4550 3430}%
\special{pa 4240 3740}%
\special{fp}%
\special{pa 4590 3450}%
\special{pa 4290 3750}%
\special{fp}%
\special{pa 4590 3510}%
\special{pa 4330 3770}%
\special{fp}%
\special{pa 4590 3570}%
\special{pa 4380 3780}%
\special{fp}%
\special{pa 4590 3630}%
\special{pa 4430 3790}%
\special{fp}%
\special{pa 4590 3690}%
\special{pa 4480 3800}%
\special{fp}%
\special{pa 4590 3750}%
\special{pa 4530 3810}%
\special{fp}%
\special{pa 4210 3290}%
\special{pa 3950 3550}%
\special{fp}%
\special{pa 4180 3260}%
\special{pa 3910 3530}%
\special{fp}%
\special{pa 4150 3230}%
\special{pa 3880 3500}%
\special{fp}%
\special{pa 4120 3200}%
\special{pa 3860 3460}%
\special{fp}%
\special{pa 4090 3170}%
\special{pa 3830 3430}%
\special{fp}%
\special{pa 4070 3130}%
\special{pa 3800 3400}%
\special{fp}%
\special{pa 4050 3090}%
\special{pa 3780 3360}%
\special{fp}%
\special{pa 4030 3050}%
\special{pa 3760 3320}%
\special{fp}%
\special{pa 4000 3020}%
\special{pa 3730 3290}%
\special{fp}%
\special{pa 3940 3020}%
\special{pa 3710 3250}%
\special{fp}%
\special{pa 3880 3020}%
\special{pa 3690 3210}%
\special{fp}%
\special{pa 3820 3020}%
\special{pa 3680 3160}%
\special{fp}%
\special{pa 3760 3020}%
\special{pa 3660 3120}%
\special{fp}%
\special{pa 3700 3020}%
\special{pa 3650 3070}%
\special{fp}%
% STR 2 0 3 0 Black White  
% 4 4570 2750 4570 2850 2 0 0 0
% $\bullet$
\put(45.7000,-28.5000){\makebox(0,0)[lb]{$\bullet$}}%
% STR 2 0 3 0 Black White  
% 4 1760 1670 1760 1770 2 0 0 0
% $\scriptstyle A_{k}$
\put(17.6000,-17.7000){\makebox(0,0)[lb]{$\scriptstyle A_{k}$}}%
% STR 2 0 3 0 Black White  
% 4 2960 2280 2960 2380 2 0 0 0
% $\scriptstyle A_{k+1}$
\put(29.6000,-23.8000){\makebox(0,0)[lb]{$\scriptstyle A_{k+1}$}}%
% STR 2 0 3 0 Black White  
% 4 1750 3840 1750 3940 2 0 0 0
% $\scriptstyle A_{k+n}$
\put(17.5000,-39.4000){\makebox(0,0)[lb]{$\scriptstyle A_{k+n}$}}%
% STR 2 0 3 0 Black White  
% 4 1190 3740 1190 3840 2 0 0 0
% $\scriptstyle A_{k+n+1}$
\put(11.9000,-38.4000){\makebox(0,0)[lb]{$\scriptstyle A_{k+n+1}$}}%
% STR 2 0 3 0 Black White  
% 4 840 2960 840 3060 2 0 0 0
% $\scriptstyle F_{1}$
\put(8.4000,-30.6000){\makebox(0,0)[lb]{$\scriptstyle F_{1}$}}%
% STR 2 0 3 0 Black White  
% 4 2160 3840 2160 3940 2 0 0 0
% $\scriptstyle G_{1}$
\put(21.6000,-39.4000){\makebox(0,0)[lb]{$\scriptstyle G_{1}$}}%
% STR 2 0 3 0 Black White  
% 4 2160 1670 2160 1770 2 0 0 0
% $\scriptstyle C_{1}$
\put(21.6000,-17.7000){\makebox(0,0)[lb]{$\scriptstyle C_{1}$}}%
% STR 2 0 3 0 Black White  
% 4 2010 2590 2010 2690 2 0 0 0
% $\scriptstyle O$
\put(20.1000,-26.9000){\makebox(0,0)[lb]{$\scriptstyle O$}}%
% STR 2 0 3 0 Black White  
% 4 2260 3390 2260 3490 2 0 0 0
% $\scriptstyle \tilde{G}_{1}$
\put(22.6000,-34.9000){\makebox(0,0)[lb]{$\scriptstyle \tilde{G}_{1}$}}%
% STR 2 0 3 0 Black White  
% 4 1500 2950 1500 3050 2 0 0 0
% $\scriptstyle \tilde{F}_{1}$
\put(15.0000,-30.5000){\makebox(0,0)[lb]{$\scriptstyle \tilde{F}_{1}$}}%
% STR 2 0 3 0 Black White  
% 4 1820 3180 1820 3280 2 0 0 0
% $\scriptstyle P$
\put(18.2000,-32.8000){\makebox(0,0)[lb]{$\scriptstyle P$}}%
% STR 2 0 3 0 Black White  
% 4 1850 2270 1850 2370 2 0 0 0
% $\scriptstyle \tilde{A}_{k}$
\put(18.5000,-23.7000){\makebox(0,0)[lb]{$\scriptstyle \tilde{A}_{k}$}}%
% STR 2 0 3 0 Black White  
% 4 2160 2290 2160 2390 2 0 0 0
% $\scriptstyle \tilde{C}_{1}$
\put(21.6000,-23.9000){\makebox(0,0)[lb]{$\scriptstyle \tilde{C}_{1}$}}%
% STR 2 0 3 0 Black White  
% 4 2650 2680 2650 2780 2 0 0 0
% $\scriptstyle \tilde{A}_{k+1}$
\put(26.5000,-27.8000){\makebox(0,0)[lb]{$\scriptstyle \tilde{A}_{k+1}$}}%
% STR 2 0 3 0 Black White  
% 4 1770 3310 1770 3410 2 0 0 0
% $\bullet$
\put(17.7000,-34.1000){\makebox(0,0)[lb]{$\bullet$}}%
% STR 2 0 3 0 Black White  
% 4 4560 1660 4560 1760 2 0 0 0
% $\scriptstyle F_{l}$
\put(45.6000,-17.6000){\makebox(0,0)[lb]{$\scriptstyle F_{l}$}}%
% STR 2 0 3 0 Black White  
% 4 4780 1700 4780 1800 2 0 0 0
% $\scriptstyle C_{l-1}$
\put(47.8000,-18.0000){\makebox(0,0)[lb]{$\scriptstyle C_{l-1}$}}%
% STR 2 0 3 0 Black White  
% 4 5660 2760 5660 2860 2 0 0 0
% $\scriptstyle G_{l}$
\put(56.6000,-28.6000){\makebox(0,0)[lb]{$\scriptstyle G_{l}$}}%
% STR 2 0 3 0 Black White  
% 4 4780 3840 4780 3940 2 0 0 0
% $\scriptstyle G_{l-1}$
\put(47.8000,-39.4000){\makebox(0,0)[lb]{$\scriptstyle G_{l-1}$}}%
% STR 2 0 3 0 Black White  
% 4 4560 3860 4560 3960 2 0 0 0
% $\scriptstyle C_{l}$
\put(45.6000,-39.6000){\makebox(0,0)[lb]{$\scriptstyle C_{l}$}}%
% STR 2 0 3 0 Black White  
% 4 3350 2950 3350 3050 2 0 0 0
% $\scriptstyle F_{l-1}$
\put(33.5000,-30.5000){\makebox(0,0)[lb]{$\scriptstyle F_{l-1}$}}%
% STR 2 0 3 0 Black White  
% 4 4110 2950 4110 3050 2 0 0 0
% $\scriptstyle \tilde{F}_{l-1}$
\put(41.1000,-30.5000){\makebox(0,0)[lb]{$\scriptstyle \tilde{F}_{l-1}$}}%
% STR 2 0 3 0 Black White  
% 4 4430 3250 4430 3350 2 0 0 0
% $\scriptstyle \tilde{C}_{l}$
\put(44.3000,-33.5000){\makebox(0,0)[lb]{$\scriptstyle \tilde{C}_{l}$}}%
% STR 2 0 3 0 Black White  
% 4 4870 3420 4870 3520 2 0 0 0
% $\scriptstyle \tilde{G}_{l-1}$
\put(48.7000,-35.2000){\makebox(0,0)[lb]{$\scriptstyle \tilde{G}_{l-1}$}}%
% STR 2 0 3 0 Black White  
% 4 5250 2880 5250 2980 2 0 0 0
% $\scriptstyle \tilde{G}_{l}$
\put(52.5000,-29.8000){\makebox(0,0)[lb]{$\scriptstyle \tilde{G}_{l}$}}%
% STR 2 0 3 0 Black White  
% 4 4440 2310 4440 2410 2 0 0 0
% $\scriptstyle \tilde{F}_{l}$
\put(44.4000,-24.1000){\makebox(0,0)[lb]{$\scriptstyle \tilde{F}_{l}$}}%
% STR 2 0 3 0 Black White  
% 4 4840 2450 4840 2550 2 0 0 0
% $\scriptstyle \tilde{C}_{l-1}$
\put(48.4000,-25.5000){\makebox(0,0)[lb]{$\scriptstyle \tilde{C}_{l-1}$}}%
% STR 2 0 3 0 Black White  
% 4 4650 2600 4650 2700 2 0 0 0
% $\scriptstyle O$
\put(46.5000,-27.0000){\makebox(0,0)[lb]{$\scriptstyle O$}}%
% STR 2 0 3 0 Black White  
% 4 1700 4100 1700 4200 2 0 0 0
% {\bf Figure 4}
\put(17.0000,-42.0000){\makebox(0,0)[lb]{{\bf Figure 4}}}%
% STR 2 0 3 0 Black White  
% 4 4320 4090 4320 4190 2 0 0 0
% {\bf Figure 5}
\put(43.2000,-41.9000){\makebox(0,0)[lb]{{\bf Figure 5}}}%
\end{picture}}%
\]

Finally,
we introduce Pizza Theorem.
It was proposed by L.~J.~Upton in Mathematics Magazine almost fifty years ago (cf.~\cite{Upton}).
In \cite{Goldberg},
Goldberg first proved this theorem and the results are widely introduced throughout the literature (cf.~\cite{Berz1}, \cite{Berz2}, \cite{Hart}, \cite{Hess}, \cite{Hirs}, \cite{Mabry}, \cite{Nelsen}, and \cite{Pearce}).
Present the statement of this theorem (see {\bf Figure 6}).
%Throughout this paper, we will discuss under the following conditions.
Let $P$ be a point of the disk $D$, and  let $n$ be a multiple of $4$. 
Next,
divide the disk $D$ into $2n$ slices by making $n$ straight lines through $P$, where the lines meet to form $2n$ equal angles.
Number the slices consecutively in a clockwise direction. 
Then,
we obtain the following theorem.
\begin{thm}[Pizza Theorem \cite{Upton}]\label{thm:Pizza Theorem}
%The sum of the areas of the odd-numbered slices equals the sum of the areas of the even-numbered slices.
We can distribute the slices to $\frac{n}{2}$ people so that the sum of the areas is equal.
\end{thm}
\begin{proof}
%We give the $k$-th person four slices of $Sl(A_k)_{P}, Sl(A_{k+\frac{n}{2}})_{P}, Sl(A_{k+n})_{P}$ and $Sl(A_{k+\frac{3n}{2}})_{P}$.
Let $\tilde{D}$ be the disk whose radius is the line segment $OP$,
and let $B$ be the Baumkuchen given by the set $D \setminus {\rm Int}\ \tilde{D}$.
Then,
we deduce the Pizza Theorem by applying Lemmas \ref{lemma2} and \ref{lemma3}.
In particular,
%by the condition $(3)$ of Lemma \ref{lemma3},
we give the $k$-th person four slices of $Sl(A_k)_{P},\ Sl(A_{k+\frac{n}{2}})_{P},\ Sl(A_{k+n})_{P}$
and $Sl(A_{k+\frac{3n}{2}})_{P}$
by the condition $(3)$ of Lemma \ref{lemma3}.
\end{proof}
\[
%WinTpicVersion4.32a
{\unitlength 0.1in%
\begin{picture}(19.6000,23.2000)(16.2000,-41.3000)%
% CIRCLE 2 0 3 0 Black White  
% 4 2600 2790 2600 1810 2600 1810 2610 1820
% 
\special{pn 8}%
\special{ar 2600 2790 980 980 4.7226979 4.7123890}%
% LINE 2 0 3 0 Black White  
% 4 2220 3700 2200 3700 2200 3690 2200 1900
% 
\special{pn 8}%
\special{pa 2220 3700}%
\special{pa 2200 3700}%
\special{fp}%
\special{pa 2200 3690}%
\special{pa 2200 1900}%
\special{fp}%
% LINE 2 0 3 0 Black White  
% 2 1650 3010 3560 3010
% 
\special{pn 8}%
\special{pa 1650 3010}%
\special{pa 3560 3010}%
\special{fp}%
% LINE 2 0 3 0 Black White  
% 2 1820 3390 3210 2030
% 
\special{pn 8}%
\special{pa 1820 3390}%
\special{pa 3210 2030}%
\special{fp}%
% LINE 2 0 3 0 Black White  
% 2 2910 3720 1670 2490
% 
\special{pn 8}%
\special{pa 2910 3720}%
\special{pa 1670 2490}%
\special{fp}%
% STR 2 0 3 0 Black White  
% 4 2250 2730 2250 2830 2 0 0 0
% $\scriptstyle P$
\put(22.5000,-28.3000){\makebox(0,0)[lb]{$\scriptstyle P$}}%
% LINE 3 0 3 0 Black White  
% 44 3490 2390 2870 3010 3470 2350 2810 3010 3450 2310 2750 3010 3430 2270 2690 3010 3400 2240 2630 3010 3380 2200 2570 3010 3350 2170 2510 3010 3320 2140 2450 3010 3300 2100 2390 3010 3260 2080 2330 3010 3230 2050 2270 3010 3510 2430 2930 3010 3520 2480 2990 3010 3540 2520 3050 3010 3550 2570 3110 3010 3560 2620 3170 3010 3570 2670 3230 3010 3580 2720 3290 3010 3580 2780 3350 3010 3580 2840 3410 3010 3570 2910 3470 3010 3560 2980 3530 3010
% 
\special{pn 4}%
\special{pa 3490 2390}%
\special{pa 2870 3010}%
\special{fp}%
\special{pa 3470 2350}%
\special{pa 2810 3010}%
\special{fp}%
\special{pa 3450 2310}%
\special{pa 2750 3010}%
\special{fp}%
\special{pa 3430 2270}%
\special{pa 2690 3010}%
\special{fp}%
\special{pa 3400 2240}%
\special{pa 2630 3010}%
\special{fp}%
\special{pa 3380 2200}%
\special{pa 2570 3010}%
\special{fp}%
\special{pa 3350 2170}%
\special{pa 2510 3010}%
\special{fp}%
\special{pa 3320 2140}%
\special{pa 2450 3010}%
\special{fp}%
\special{pa 3300 2100}%
\special{pa 2390 3010}%
\special{fp}%
\special{pa 3260 2080}%
\special{pa 2330 3010}%
\special{fp}%
\special{pa 3230 2050}%
\special{pa 2270 3010}%
\special{fp}%
\special{pa 3510 2430}%
\special{pa 2930 3010}%
\special{fp}%
\special{pa 3520 2480}%
\special{pa 2990 3010}%
\special{fp}%
\special{pa 3540 2520}%
\special{pa 3050 3010}%
\special{fp}%
\special{pa 3550 2570}%
\special{pa 3110 3010}%
\special{fp}%
\special{pa 3560 2620}%
\special{pa 3170 3010}%
\special{fp}%
\special{pa 3570 2670}%
\special{pa 3230 3010}%
\special{fp}%
\special{pa 3580 2720}%
\special{pa 3290 3010}%
\special{fp}%
\special{pa 3580 2780}%
\special{pa 3350 3010}%
\special{fp}%
\special{pa 3580 2840}%
\special{pa 3410 3010}%
\special{fp}%
\special{pa 3570 2910}%
\special{pa 3470 3010}%
\special{fp}%
\special{pa 3560 2980}%
\special{pa 3530 3010}%
\special{fp}%
% LINE 3 0 3 0 Black White  
% 44 2590 3410 2290 3710 2620 3440 2330 3730 2650 3470 2380 3740 2680 3500 2430 3750 2710 3530 2480 3760 2740 3560 2530 3770 2770 3590 2590 3770 2800 3620 2650 3770 2830 3650 2720 3760 2860 3680 2790 3750 2560 3380 2240 3700 2530 3350 2200 3680 2500 3320 2200 3620 2470 3290 2200 3560 2440 3260 2200 3500 2410 3230 2200 3440 2380 3200 2200 3380 2350 3170 2200 3320 2320 3140 2200 3260 2290 3110 2200 3200 2260 3080 2200 3140 2230 3050 2200 3080
% 
\special{pn 4}%
\special{pa 2590 3410}%
\special{pa 2290 3710}%
\special{fp}%
\special{pa 2620 3440}%
\special{pa 2330 3730}%
\special{fp}%
\special{pa 2650 3470}%
\special{pa 2380 3740}%
\special{fp}%
\special{pa 2680 3500}%
\special{pa 2430 3750}%
\special{fp}%
\special{pa 2710 3530}%
\special{pa 2480 3760}%
\special{fp}%
\special{pa 2740 3560}%
\special{pa 2530 3770}%
\special{fp}%
\special{pa 2770 3590}%
\special{pa 2590 3770}%
\special{fp}%
\special{pa 2800 3620}%
\special{pa 2650 3770}%
\special{fp}%
\special{pa 2830 3650}%
\special{pa 2720 3760}%
\special{fp}%
\special{pa 2860 3680}%
\special{pa 2790 3750}%
\special{fp}%
\special{pa 2560 3380}%
\special{pa 2240 3700}%
\special{fp}%
\special{pa 2530 3350}%
\special{pa 2200 3680}%
\special{fp}%
\special{pa 2500 3320}%
\special{pa 2200 3620}%
\special{fp}%
\special{pa 2470 3290}%
\special{pa 2200 3560}%
\special{fp}%
\special{pa 2440 3260}%
\special{pa 2200 3500}%
\special{fp}%
\special{pa 2410 3230}%
\special{pa 2200 3440}%
\special{fp}%
\special{pa 2380 3200}%
\special{pa 2200 3380}%
\special{fp}%
\special{pa 2350 3170}%
\special{pa 2200 3320}%
\special{fp}%
\special{pa 2320 3140}%
\special{pa 2200 3260}%
\special{fp}%
\special{pa 2290 3110}%
\special{pa 2200 3200}%
\special{fp}%
\special{pa 2260 3080}%
\special{pa 2200 3140}%
\special{fp}%
\special{pa 2230 3050}%
\special{pa 2200 3080}%
\special{fp}%
% LINE 3 0 3 0 Black White  
% 16 2030 3010 1760 3280 2090 3010 1780 3320 2150 3010 1810 3350 1970 3010 1740 3240 1910 3010 1720 3200 1850 3010 1700 3160 1790 3010 1680 3120 1730 3010 1670 3070
% 
\special{pn 4}%
\special{pa 2030 3010}%
\special{pa 1760 3280}%
\special{fp}%
\special{pa 2090 3010}%
\special{pa 1780 3320}%
\special{fp}%
\special{pa 2150 3010}%
\special{pa 1810 3350}%
\special{fp}%
\special{pa 1970 3010}%
\special{pa 1740 3240}%
\special{fp}%
\special{pa 1910 3010}%
\special{pa 1720 3200}%
\special{fp}%
\special{pa 1850 3010}%
\special{pa 1700 3160}%
\special{fp}%
\special{pa 1790 3010}%
\special{pa 1680 3120}%
\special{fp}%
\special{pa 1730 3010}%
\special{pa 1670 3070}%
\special{fp}%
% LINE 3 0 3 0 Black White  
% 40 2200 2300 1840 2660 2200 2360 1870 2690 2200 2420 1900 2720 2200 2480 1930 2750 2200 2540 1960 2780 2200 2600 1990 2810 2200 2660 2020 2840 2200 2720 2050 2870 2200 2780 2080 2900 2200 2840 2110 2930 2200 2900 2140 2960 2200 2960 2170 2990 2200 2240 1810 2630 2200 2180 1780 2600 2200 2120 1750 2570 2200 2060 1720 2540 2200 2000 1690 2510 2200 1940 1680 2460 2160 1920 1730 2350 1980 2040 1850 2170
% 
\special{pn 4}%
\special{pa 2200 2300}%
\special{pa 1840 2660}%
\special{fp}%
\special{pa 2200 2360}%
\special{pa 1870 2690}%
\special{fp}%
\special{pa 2200 2420}%
\special{pa 1900 2720}%
\special{fp}%
\special{pa 2200 2480}%
\special{pa 1930 2750}%
\special{fp}%
\special{pa 2200 2540}%
\special{pa 1960 2780}%
\special{fp}%
\special{pa 2200 2600}%
\special{pa 1990 2810}%
\special{fp}%
\special{pa 2200 2660}%
\special{pa 2020 2840}%
\special{fp}%
\special{pa 2200 2720}%
\special{pa 2050 2870}%
\special{fp}%
\special{pa 2200 2780}%
\special{pa 2080 2900}%
\special{fp}%
\special{pa 2200 2840}%
\special{pa 2110 2930}%
\special{fp}%
\special{pa 2200 2900}%
\special{pa 2140 2960}%
\special{fp}%
\special{pa 2200 2960}%
\special{pa 2170 2990}%
\special{fp}%
\special{pa 2200 2240}%
\special{pa 1810 2630}%
\special{fp}%
\special{pa 2200 2180}%
\special{pa 1780 2600}%
\special{fp}%
\special{pa 2200 2120}%
\special{pa 1750 2570}%
\special{fp}%
\special{pa 2200 2060}%
\special{pa 1720 2540}%
\special{fp}%
\special{pa 2200 2000}%
\special{pa 1690 2510}%
\special{fp}%
\special{pa 2200 1940}%
\special{pa 1680 2460}%
\special{fp}%
\special{pa 2160 1920}%
\special{pa 1730 2350}%
\special{fp}%
\special{pa 1980 2040}%
\special{pa 1850 2170}%
\special{fp}%
% STR 2 0 3 0 Black White  
% 4 2420 3910 2420 4010 2 0 0 0
% $n=4$
\put(24.2000,-40.1000){\makebox(0,0)[lb]{$n=4$}}%
% STR 2 0 3 0 Black White  
% 4 2280 4160 2280 4260 2 0 0 0
% {\bf Figure 6}
\put(22.8000,-42.6000){\makebox(0,0)[lb]{{\bf Figure 6}}}%
\end{picture}}%
\]

%
%
%
%
% Refferences


\begin{thebibliography}{99}
\bibitem{Berz1}
G.~Berzsenyi,
\emph{
The pizza theorem-part I,
}
Quantum {\bf 4}(January/February, 1994)29.
%
%
%
%
%
\bibitem{Berz2}
G.~Berzsenyi,
\emph{
The pizza theorem-part II,
}
Quantum {\bf 4}(March/April, 1994) 29.
%
%
%
%
%
%
\bibitem{Goldberg}
M.~Goldberg,
\emph{
Divisors of a circle (solution to problem 660),
}
Math. Mag.
{\bf 41}(1968)46.
%
%
%
%
%
\bibitem{Hart}
J.~H$\ddot{\rm a}$rterich and S.~Iwata,
\emph{
Solutions to problem 1325,
}
Crux Math.
{\bf 15}(1989)120–122.
%
%
%
%
\bibitem{Hess}
R.~I.~Hess,
\emph{
Solution to problem 1535,
}
Crux Math.
{\bf 17}(1991)179–181.
%
%
%
%
\bibitem{Hirs}
J.~Hirschhorn,
M.~Hirschhorn,
J.~K.~Hirschhorn,
A. Hirschhorn and P.~Hirschhorn, 
\emph{
The pizza theorem,
}
Austral.~Math.~Soc.~Gaz.
{\bf 26}(1999)120–121.
%
%
%
%
%
\bibitem{Mabry}
Rick Mabry and Paul Deiermann,
\emph{
Of Cheese and Crust: A Proof of the Pizza Conjecture and Other Tasty Results,
}
American Mathematical Monthly,
{\bf 116}(2009)423–438.
%
%
%
%
\bibitem{Nelsen}
R.~B.~Nelsen,
\emph{Proof without Words: Four squares with constant area,}
Math. Mag. {\bf 77}(2004)135.
%
%
%
%
\bibitem{Pearce}
C.~E.~M.~Pearce,
\emph{
More on the pizza theorem,
}
Austral.~Math.~Soc.~Gaz. {\bf 27}(2000)4–5.
\bibitem{Upton}
L.~J.~Upton,
\emph{
Problem 660,
}
Math.~Mag.
{\bf 40}(1967)163.
\end{thebibliography}
\end{document}